\documentclass{amsart}

\usepackage{amssymb}
\usepackage{mathrsfs}
\usepackage{graphicx}
\usepackage{comment}
\usepackage{url}

\newcommand{\ve}{\varepsilon}

\newcommand{\be}{\begin{equation}}
\newcommand{\ee}{\end{equation}}
\newcommand{\ba}{\begin{align}}
\newcommand{\ea}{\end{align}}

\newcommand{\abs}[1]{\lvert#1\rvert}

\newtheorem{theorem}{Theorem}[section]
\newtheorem{corollary}[theorem]{Corollary}
\newtheorem{lemma}{Lemma}[section]

{\begin{list}{}{%
\settowidth{\labelwidth}{\textsf{{\it #1.}}}%
\setlength{\labelsep}{4mm}%
\setlength{\leftmargin}{\labelwidth}%
\addtolength{\leftmargin}{\labelsep}%
}}%
{\end{list}}

{\begin{list}{}{%
\settowidth{\labelwidth}{\textsf{{\it #1.}}}%
\setlength{\labelsep}{2mm}%
\setlength{\leftmargin}{\labelwidth}%
\addtolength{\leftmargin}{\labelsep}%
\addtolength{\leftmargin}{4mm}%
\setlength{\itemsep}{6pt}%
\setlength{\listparindent}{0pt}%
\setlength{\topsep}{3pt}%
}}%
{\end{list}}

\title[Heisenberg group determinants]{The Integer group determinants for the Heisenberg group of order $p^3$}

\author[M. Mossinghoff]{Michael J.  Mossinghoff}
\address{Center for Communications Research, Princeton, NJ, USA}
\email{m.mossinghoff@idaccr.org}

\author[C. Pinner]{Christopher Pinner}
\address{Department of Mathematics\\
         Kansas State University\\
         Manhattan, KS 66506, USA}
\email{pinner@math.ksu.edu}

\keywords{Group determinant, Heisenberg group, Lind-Lehmer constant, Mahler measure}
\subjclass[2010]{Primary: 11C20, 15B36; Secondary:  11B83, 11C08,  11G50, 11R06, 11R09, 11T22, 20H25, 20H30, 43A40}
\date{\today}

\begin{document}

\begin{abstract}
We establish a congruence satisfied by the integer group determinants for the non-abelian Heisenberg group of order $p^3$.
We characterize all determinant values coprime to $p$, give sharp divisibility conditions for multiples of $p$, and determine all values when $p=3$.
We also provide new sharp conditions on the power of $p$ dividing  the group determinants for $\mathbb Z_p^2$. 

For a finite group, the integer group determinants can be understood as corresponding to Lind's generalization of the Mahler measure. We speculate on the Lind-Mahler measure for  the discrete Heisenberg group and for two other infinite non-abelian groups
arising from symmetries of the plane and 3-space.
\end{abstract}

\maketitle

\section{Introduction}

Recall that for a group $G=\{ g_1,\ldots ,g_k\}$ and variables $a_{g_1},\ldots ,a_{g_k}$, the \textit{group determinant} of $G$ is the homogeneous polynomial of degree $k$ in the $a_{g_i}$ arising from the  $k\times k$ determinant  
\be \label{defDet} \mathscr{D}(a_{g_1},\ldots ,a_{g_k})=\det \left( a_{g_ig_j^{-1}}\right),\ee
where $i=1,\ldots ,k$ indexes rows and $j=1,\ldots ,k$ columns. The group determinant determines the group \cite{Formanek}.
Taussky-Todd asked what values the group determinant takes when the $a_g$ are all integers,
\be \label{DefS} \mathscr{S}(G)=\{  \mathscr{D}(a_{g_1},\ldots ,a_{g_k})\; : \; a_{g_1},\ldots ,a_{g_n}\in \mathbb Z\}, \ee
with particular interest  in the cyclic groups, where these are circulant determinants. For a $p$-group it is often convenient 
to write
$$ \mathscr{S}(G)= \mathscr{S}_0 (G)\cup \mathscr{S}_1(G), $$
where $\mathscr{S}_1(G)$ are the values coprime to $p$ and  $\mathscr{S}_0(G)$ are the multiples of $p$.
For example, Laquer \cite{Laquer} and Newman \cite{Newman1} studied this question for cyclic $p$-groups, showing that 
$$ \mathscr{S}_1(\mathbb Z_{p^k})=\{m\in \mathbb Z\; :\; p\nmid m\} $$
for $k\geq 1$, as well as
\be \label{newman} \mathscr{S}_0(\mathbb Z_p)=p^2\mathbb Z, \ee
and
$$p^{2k}\mathbb Z\subseteq \mathscr{S}_0(\mathbb Z_{p^k})\subseteq p^{k+1}\mathbb Z $$
for $k\geq 2$, with $\mathscr{S}_0(\mathbb Z_9)=3^3\mathbb Z$ but $\mathscr{S}_0(\mathbb Z_{p^2})\neq  p^3\mathbb Z$ for $p\geq 5$.
From \cite{smallgps}
\be \label{Z3^2}  \mathscr{S}_0(\mathbb Z_3^2)=3^6\mathbb Z,\;\;\; \mathscr{S}_1(\mathbb Z_3^2)=\{m\;:\; m\equiv \pm 1 \bmod 9\}. \ee

For a finite abelian group the group determinant factors into linear factors and coincides with Lind's generalization of the Mahler 
measure to an arbitrary compact abelian group \cite{Lind}.
This is defined  in terms of the characters of the group, with the classical Mahler measure corresponding to the group $\mathbb R/\mathbb Z$, see for example \cite{Norbert,Norbert2}.
The Lind-Lehmer problem for the group then corresponds to finding the smallest non-trivial element of $\mathscr{S}(G)$. Results of this type for $G=\mathbb Z_n$, $\mathbb Z_p^k$ and $\mathbb Z_p \times \mathbb Z_{p^2}$ can be found respectively in \cite{Pigno1}, \cite{dilum} and \cite{pgroups}.

For a finite group $G$ we can use the group determinant  to define a Mahler-type measure on the elements $F=\sum_{g\in G} a_g g$ of the group ring $\mathbb Z [G]$:
$$ m_G\left(F \right):=\frac{1}{|G|}\log M_G(F),\;\; M_G(F):=\mathscr{D}(a_{g_1},\ldots ,a_{g_n}).$$
This extends the concept of Lind measure to  non-abelian finite groups. The elements of $\mathbb Z [G]$  can be associated with elements from a  polynomial ring, with non-commuting monomials in the case of a non-abelian group, and the usual multiplicative 
property $M_G(f_1,f_2)=M_G(f_1)M_G(f_2)$ of the Mahler measure  if we multiply and reduce polynomials using the group relations on the monomials. The Lind-Lehmer problem for the group then  becomes to determine
\be \label{LindLehmer} \lambda(G):= \frac{1}{|G|} \log \min \left\{ |M_G(F)|\geq 2 \; : \; F\in \mathbb Z [G]\right\}. \ee
Alternative approaches can be found in \cite{Lalin,Lind2,luck}.
The case of the dihedral group was considered in \cite{dihedral}, and a complete description of the integer group determinants for all groups of order at most 14 was obtained in \cite{smallgps}.

Recall that for a commutative ring with unity $R$, the \textit{Heisenberg group} over $R$, denoted $H(R)$, is the group of upper triangular $3\times 3$ matrices under matrix multiplication:
$$ H(R)=\left\{ \begin{pmatrix} 1 & a & c \\ 0 & 1 & b \\ 0 & 0 & 1 \end{pmatrix} \; : \; a,b,c \in R\right\}. $$
Here we are concerned with the discrete Heisenberg group $H=H(\mathbb Z)$, an  infinite non-abelian group,
and with $H_p=H(\mathbb Z_p)$ one of the two non-abelian groups of order $p^3$. 
Recall that there are five  groups of order $p^3$: the three abelian groups $\mathbb Z_{p^3}$ (see \cite{Norbert2,Newman2}), $\mathbb Z_p^3$ (see \cite{dilum}) and $\mathbb Z_p \times \mathbb Z_{p^2}$ (see \cite{pgroups}), and the two non-abelian groups $H_p$ and (the yet to be explored)  $\mathbb Z_p \rtimes \mathbb Z_{p^2}$.
We assume throughout that $p\geq 3$; the group  $H_2$ corresponds to the dihedral group $D_8$, and a complete description
of its group determinants was obtained in \cite{dihedral}:
$$ \mathscr{S}_1(D_8)=\{ 4m+1\;:\; m\in \mathbb Z\},\;\; \mathscr{S}_0(D_8)= 2^8\mathbb Z. $$

This paper is organized in the following way.
In Section~\ref{dilum} we characterize the values of $\mathscr{S}_0(\mathbb{Z}_p^2)$, as this is needed in our study of group determinants for the Heisenberg groups $H_p$.
Section~\ref{secHeisenberg} defines the $H_p$ measure of a polynomial and states our main results concerning the sets $\mathscr{S}_0(H_p)$ and  $\mathscr{S}_1(H_p)$.
Proofs of these results appear in Section~\ref{Proofs}.
Finally, Section~\ref{secInfinite} speculates on the Lind-Mahler measure for three infinite non-abelian groups: the  discrete Heisenberg group $H(\mathbb{Z})$, the group of rotations and reflections of the circle $D_\infty$, and  $D_{\infty h}$ when
reflections through the origin are  added to $D_{\infty}$.

\section{The group $G=\mathbb Z_p \times \mathbb Z_p$}\label{dilum}

For $G=\mathbb Z_p^n$ the group determinant $\mathscr{D}_G(a_{g_1},\ldots ,a_{g_{p^n}})$  takes the form
$$ M_G(F)=\prod_{j_1=0}^{p-1} \cdots \prod_{j_n=0}^{p-1} F(w^{j_1},\ldots ,w^{j_n}),\;\; w:=e^{2\pi i/p}, $$
where
$$F(x_1,\ldots ,x_n)=\sum_{g=(m_1,\ldots ,m_n)\in G} a_g\: x_1^{m_1}\cdots x_n^{m_n} \in \mathbb Z [x_1,\ldots ,x_n]. $$
In \cite{dilum} it was shown that
$ \mathscr{S}_1(\mathbb Z_p^n)$ is exactly  the set
\begin{align}\label{defT}
\mathscr{T}_n &:=  \{m\in \mathbb Z\; : \; m\equiv a^{p^{n-1}} \bmod p^n, \text{ for some } 1\leq a <p\} \\\nonumber
&= \{x\in \mathbb Z\; : \; x^{p-1}\equiv 1 \bmod p^n\}.
\end{align}
 For example, we get
$$ \mathscr{S}_1(\mathbb Z_3^n)=\{m\; : \; m\equiv \pm 1 \bmod 3^n\}. $$
Writing $u:=1-w$ we have $|u|_p=p^{-1/(p-1)}$ and $F(w^{j_1},\ldots ,w^{j_n})\equiv F(1,\ldots ,1)$ mod $u$, and $p\mid M_G(F)$
if and only if $p\mid F(1,\ldots ,1)$, with $u$ dividing all  the other terms. We immediately obtain the rough bound
\be \label{rough}  p\mid M_G(F)\;\; \Rightarrow \;\;p^{1+(p^n-1)/(p-1)}  \mid M_G(F). \ee
This was enough in \cite{dilum}, where the primary concern was the smallest non-trivial measure (the values divisible by $p$ are clearly larger than the trivial bound  $|G|-1$), and from \eqref{newman} is plainly sharp when $n=1$.
It can be improved though when $n\geq 2$.
To obtain the optimal $p$ divisibility result for $G=H_p$ we need the optimal result for $G=\mathbb Z_p^2$, which we establish in the following theorem.

\begin{theorem} \label{Zp} If $G=\mathbb Z_p^2$ and $p\mid M_{G}(F)$ then $p^{p+3}\mid M_{G}(F)$.
Moreover, for any $k\geq 0$ there exists a polynomial $F\in\mathbb Z[x,y]$ with $p^{p+3+k}\parallel M_{G}(F).$ 
\end{theorem}

\begin{proof} We  expand $F$ and write
$$ F(x,y)=A+B(1-x)+C(1-y) +\sum_{i+j\geq 2} A_{ij}(1-x)^i(1-y)^j. $$
Plainly  $F(\omega^i,\omega^j)\equiv A$ mod $u$ and as above $p\mid M_G(F)$ forces $p\mid A=F(1,1),$ $F(\omega^i,\omega^j)\equiv 0$ mod $u$ and $p^{1+(p^2-1)/(p-1)}\mid M_G(F)$. As $M_G(F)$ is an integer, to improve the exponent from $p+2$ to $p+3$ it will be enough to show that $F(\omega^i,\omega^j)\equiv 0$ mod $u^2$ for
at least one of the $(i,j)\neq (0,0)$. From $\omega^i=(1-u)^i\equiv 1-iu$ mod $u^2,$ 
$$ F(\omega^i,\omega^j) \equiv A+B(1-\omega^i)+C(1-\omega^j) \equiv (Bi+Cj)(1-\omega) \text{ mod $u^2$}. $$
Hence if $p\mid B$ we can take $j=0$ and any $i=1,\ldots ,p-1$, and if $p\nmid B$ we can take $i\equiv -CB^{-1} j$ mod $p$ 
for any $j=1,\ldots ,p-1$ and $F(\omega^i,\omega^j)\equiv 0$ mod $u^2.$

For $F(x,y)=A_1 p^{1+k} + A_2(1-x)+A_3(1-y)^2,$ $p\nmid A_1A_2A_3,$ we have $p^{1+k}\parallel F(1,1)$,
$u\parallel F(\omega^i,\omega^j), 1\leq i <p,0\leq  j<p,$ and if $p\geq 5$, $u^2\parallel F(1,\omega^j), 1\leq j<p,$  giving $p^{1+k+p+2}\parallel M_G(F)$. For $p=3$ we know that we can achieve all $3^{6+k}$ by \eqref{Z3^2}.
\end{proof}

\section{The Heisenberg groups}\label{secHeisenberg}

Writing
$$ X=\begin{pmatrix} 1 & 0 & 0 \\ 0 & 1 & 1 \\ 0 & 0 & 1 \end{pmatrix},\;\;\;  Y=\begin{pmatrix} 1 & 1 & 0 \\ 0 & 1 & 0 \\ 0 & 0 & 1 \end{pmatrix}, \;\;\;Z=\begin{pmatrix} 1 & 0 & 1 \\ 0 & 1 & 0 \\ 0 & 0 & 1 \end{pmatrix}, $$
we have 
$$    X^aY^bZ^c=\begin{pmatrix} 1 & b & c \\ 0 & 1 & a \\ 0 & 0 & 1 \end{pmatrix}, \;\;XZ=ZX,\;\; YZ=ZY,\;\; YX=XYZ.$$
Hence we can write
$$ H_p=\{ x^ay^bz^c\; : \; 0\leq a,b,c<p,\;\; x^p=y^p=z^p=1,\; xz=zx,\;yz=zy,\; yx=xyz\}, $$
and 
$$ H=\{ x^ay^bz^c\; : \; a,b,c\in \mathbb Z,\;\; xz=zx,\;yz=zy,\; yx=xyz\}. $$

\subsection{The $H_p$ measure of a polynomial}

For $G=H_p$ we can think of the elements in $\mathbb Z [G]$ as polynomials
\be \label{H-poly}  F(x,y,z)=\sum_{0\leq i,j,k <p} a_{ijk} x^iy^jz^k, \;\; a_{ijk}\in \mathbb Z, \ee
in a ring with non-commutative variables satisfying the group relations  $x^p=y^p=z^p=1$, $xz=zx$, $yz=zy,$ $yx=xyz$.
It will be convenient to write \eqref{H-poly} in the form
$$ F(x,y,z)=\sum_{i=0}^{p-1} x^i \: f_i(y,z),\;\;\; f_i(y,z)= \sum_{0\leq j,k <p} a_{ijk} y^jz^k. $$
From Frobenius \cite{Frobenius} (see \cite{Conrad,Formanek}), we have a factorization of the group determinant
$$\mathscr{D}(a_{g_1},\ldots ,a_{g_k})=\prod_{\rho \in \hat{G}} \det  \rho\left(\sum_{g\in G} a_g g\right)^{\deg (\rho)}, $$
where $\hat{G}$ is the set of irreducible representations for $G$. 

For $G=H_p$ we have $p^2$  degree~1 characters, namely,
$$ \chi_{ij} (x)=w^i,\;\; \chi_{ij} (y)=w^j,\;\; \chi_{ij}(z)=1, \;\;\; \chi_{ij}(F)=F(w^i,w^j,1), $$
for $0\leq i,j<p$, and $(p-1)$ degree $p$ representations $\phi_{\lambda}$, $\lambda=w^j$, $1\leq j<p$, with 
$$ \phi_{\lambda} (x)=\begin{pmatrix} 0 & 0 & \cdots & 0 & 1 \\
  1 & 0 & \cdots & 0 & 0 \\
 \vdots  &  \vdots & \ddots & \vdots & \vdots \\
 0 & 0 & \cdots & 1 & 0 \end{pmatrix},\;\;\;   \phi_{\lambda} (y)=\begin{pmatrix} 1 & 0 & \cdots & 0 & 0  \\
  0 & \lambda  & \cdots & 0 & 0 \\
 \vdots  &  \vdots & \ddots & \vdots & \vdots \\
 0 & 0 & \cdots & 0 & \lambda^{p-1} \end{pmatrix},\;\;\; \phi_{\lambda}(z)=\lambda I_p. $$
It is readily checked that these do satisfy the group relations and give the desired representations.
That is,  $\phi_{\lambda} (x)$ shifts rows down by one cyclically, $\phi_{\lambda}(y)$ multiplies the successive rows 
by $1,\lambda,\ldots, \lambda^{p-1}$, and $\phi_{\lambda}(z)$ multiplies everything by $\lambda$, giving
$$ \phi_{\lambda}(F) = \begin{pmatrix} f_0(1,\lambda) & f_{p-1}(1,\lambda) & f_{p-2} (1,\lambda) & \cdots & f_1(1,\lambda) \\
  f_1(\lambda ,\lambda) & f_0(\lambda,\lambda) & f_{p-1}(\lambda ,\lambda) & \cdots & f_2(\lambda,\lambda )\\
   f_2(\lambda^2,\lambda) & f_1(\lambda^2,\lambda) & f_0(\lambda^2,\lambda) & \cdots & f_3(\lambda^2,\lambda) \\
\vdots & \vdots & \vdots & \ddots & \vdots \\
 f_{p-1}(\lambda^{p-1},\lambda) & f_{p-2}(\lambda^{p-1},\lambda) & f_{p-3}(\lambda^{p-1},\lambda) & \cdots & f_0(\lambda^{p-1},\lambda) \end{pmatrix}. $$
Thus, the $j$th row will consist of $f_{p-1}(y,\lambda),f_{p-2}(y,\lambda), \ldots , f_0(y,\lambda)$ shifted cyclically $j$ places to the right and evaluated at $y=w^{j-1}$, $j=1,\ldots ,p$. 
Hence we have group determinant
$$ M_{H_p} (F) = M_1 M_2^p  $$
where
$$ M_1=\prod_{0\leq i,j<p} \chi_{ij}(F) = \prod_{0\leq i,j <p} F(w^i,w^j,1)=M_{\mathbb Z_p^2}(F(x,y,1)) $$
and
$$ M_2=\prod_{j=1}^{p-1} D(w^j), \;\; D(\lambda):=\det(\phi_{\lambda}(F)). $$
Note that $M_1$ and $M_2$ are integers.

In the case $F(x,y,z)=f_0(y,z)$, we have 
$$ D(\lambda)=\prod_{j=0}^{p-1} f_0(w^j,\lambda), \;\; M_1=\prod_{j=0}^{p-1} f_0(w^j,1)^p,\;\; M_{H_p}(F) = M_{\mathbb Z_p^2}(f_0)^p. $$
Similarly, if $F(x,y,z)=G(x,z)=\sum_{i=0}^{p-1} x^i f_i(z)$ then the $D(\lambda)$ are circulant determinants and
$$ D(\lambda)= \prod_{i=0}^{p-1} G(w^i,\lambda), \;\; M_1=\prod_{i=0}^{p-1} G(w^i,1)^p,\;\; M_{H_p}(F)=M_{\mathbb Z_p^2}(G)^p. $$
This is as we should expect since  $\langle Y,Z\rangle$ and  $\langle X,Z\rangle$ are abelian subgroups of index $p$ in $H_p$, both of 
the form $\mathbb Z_p^2$.

For a polynomial that is binomial in $x$,
$$ F(x,y,z) = f_0(y,z)+x^kf_k(y,z), \;\;\; 1\leq k<p, $$
it is readily seen that 
\be \label{binomial}  D(\lambda) = \prod_{j=0}^{p-1} f_0(w^j,\lambda) +  \prod_{j=0}^{p-1} f_k(w^j,\lambda) \ee
and
\be \label{binomial2}   M_1=   \prod_{j=0}^{p-1} \left(   f_0(w^j,1)^p +  f_k(w^j,1)^p  \right). \ee

We show that the determinant values  for  $H_p$ satisfy the following congruence.

\begin{theorem}\label{main}  For $G=H_p$ the Heisenberg measure of a polynomial  $F(x,y,z)$ of the form \eqref{H-poly} satisfies
\be \label{cong}  M_G(F)\equiv F(1,1,1)^{p^3} \bmod p^3. \ee
\end{theorem}

Notice that reducing a general polynomial in $\mathbb Z [x,y,z]$  to the standard form \eqref{H-poly} using the Heisenberg group relations does not change the value of $F(1,1,1)$.
Theorem~\ref{main} is the same congruence satisfied when $G=\mathbb Z_p^3$, see \cite{dilum}.
Indeed, from this we deduce that the values coprime to $p$ are the same in both cases.

\begin{corollary}\label{notp}
In the notation of \eqref{defT},
$$\mathscr{S}_1(H_p)=\mathscr{T}_3=\{ x\in \mathbb Z \; : \; x^{p-1}\equiv 1 \bmod p^3\}. $$
\end{corollary}

Theorem~\ref{main} tells us that $\mathscr{S}_1(H_p)\subseteq \mathscr{T}_3$, so  Corollary~\ref{notp} follows at  once from showing that we can achieve every value in $\mathscr{T}_3$:

\begin{theorem} \label{achieve} For any $a$ in $\mathbb N$ with $p\nmid a$ and $m$ in $\mathbb Z$  there is an $F(x,y,z)$ of the form \eqref{H-poly} with $M_{H_p}(F)=a^{p^2}+mp^3.$
\end{theorem}
Of course $a^{p^3}\equiv a^{p^2}$ mod $p^3$. 
For the multiples of $p$ we have:

\begin{theorem} \label{pdiv}  If $p\mid M_{H_p}(F)$ then $p^{p^2+3}\mid M_{H_p}(F)$. 
In addition, there exists a polynomial $F$ with $p^{p^2+3}\parallel M_{H_p}(F). $

\end{theorem}
Hence while $\mathscr{S}_1(G)$, the values  coprime to $p$,  are the same for  $G=\mathbb Z_p^3$ and $G=H_p$, we know from \eqref{rough} and Theorem~\ref{pdiv} that $\mathscr{S}_0(G)$, the multiples of $p$, are different.

For the Lind-Lehmer problem \eqref{LindLehmer} we immediately obtain: 
\begin{corollary}
$$ \lambda(H_p)=\lambda (\mathbb Z_p^3) =\frac{1}{p^3} \log \min\{ x\geq 2 \; : \; x^{p-1}\equiv 1 \bmod p^3\}. $$
\end{corollary}

\subsection{The group $H_3$} 

From above we know $\mathscr{S}_1(H_3)$ and that $\mathscr{S}_0(H_3)\subseteq 3^{12}\mathbb Z$. We show that all multiples of $3^{12}$ can be achieved.
Here and throughout $\Phi_p(x)$ denotes the $p$th cyclotomic polynomial
\be \label{defPhi} \Phi_p(x) :=1+x+\cdots +x^{p-1} =(x^p-1)/(x-1). \ee

\begin{theorem} \label{p=3} Let $G=H_3$, then 
$$ \mathscr{S}_1(H_3)=\{m\; :\; m\equiv \pm 1 \bmod 27\}, \;\;\; \mathscr{S}_0(H_3)= 3^{12}\mathbb Z. $$
\end{theorem}

\begin{proof} From above 
we have $\mathscr{S}_1(H_3)$ as stated and $\mathscr{S}_0(H_3)\subseteq 3^{12}\mathbb Z$.
To see that we get all multiples of $3^{12}$, observe that  $M_G(-F)=-M_G(F)$ with
\begin{align*}
M_G(z+y-y^2 +(y+1)x +  m\Phi_3(x)\Phi_3(y)\Phi_3(z) )  & =3^{12}(1+9m), \\
M_G(1+2x+x^2\Phi_3(y) +  m\Phi_3(x)\Phi_3(y)\Phi_3(z))  & =3^{12}(2+9m),\\
M_G(1+2x+(z+x^2)\Phi_3(y)+  m\Phi_3(x)\Phi_3(y)\Phi_3(z))&=3^{13}(1+3m),\\
M_G(1+2x-x\Phi_3(y) +  m\Phi_3(x)\Phi_3(y)\Phi_3(z))  & =3^{14}m,
\end{align*}
and $ 3^{12}(4+9m)$ from the polynomials
$$ 1+y-y^2+(y+1)x+\Phi_3(x)\Phi_3(y)+(z-1)\Phi_3(y)+(z-1)^2x\Phi_3(y)  +  m\Phi_3(x)\Phi_3(y)\Phi_3(z). $$
\end{proof}

\section{Proof of Theorems~\ref{main},~\ref{achieve} and~\ref{pdiv}}\label{Proofs}

For the proof of Theorem~\ref{main} we require two preliminary results.

\begin{lemma}\label{lemma1}  Suppose $f(x)$ is in $\mathbb Z[x]$ then 
$$ \frac{1}{p} \sum_{y^p=1} f(y)^p \equiv \prod_{y^p=1} f(y) \bmod p^2. $$
\end{lemma}

\begin{proof} We write $A_1,\ldots,A_{p}$ for $f(1),f(w),\ldots ,f(w^{p-1})$. We need to show that 
$$ A_1^p+\cdots +A_p^p \equiv p A_1\cdots A_p \bmod p^3. $$
First, observe that 
$$p\mid  e_i(A_1,\ldots ,A_p), \;\; i=1,\ldots ,p-1,\;\;\; e_p(A_1,\ldots ,A_p)=A_1\cdots A_p, $$ 
where the $e_i$ are the usual elementary symmetric functions
$$1+\sum_{i=1}^p e_i(A_1,\ldots ,A_p)x^i = \prod_{i=1}^p (1+A_ix).$$ 
To see this observe  that the $A_i\equiv A_1$ mod $u$ where $u=1-w$ and hence  $e_i(A_1,\ldots ,A_p) \equiv e_i(A_1,\ldots ,A_1) = \binom{p}{i} A_1^i\equiv 0$ mod $u$.  Since these will be integers this holds  mod $p$ and will be 0 for $i=1,\ldots,p-1$. 
Likewise  for all $l$ in $\mathbb N$ the sums
$$ s_l(A_1,\ldots , A_p)=A_1^l+\cdots +A_p^l \equiv pA_1^l\equiv 0 \bmod p. $$
From the Newton-Girard formulae we have for all $1\leq k\leq p$
\begin{align*}  s_k(A_1,\ldots,A_p) &  =(-1)^{k-1}ke_k(A_1,\ldots ,A_p) +\sum_{i=1}^{k-1}(-1)^{k-1+i} s_i(A_1,\ldots ,A_p) e_{k-i}(A_1,\ldots ,A_p) \\
 & \equiv (-1)^{k-1}k e_k(A_1,\ldots ,A_p) \bmod p^2. 
\end{align*}
From this, and the fact that $p\mid  e_{p-i}(A_1,\ldots ,A_p)$, we get
\begin{align*}  s_p(A_1,\ldots,A_p) &  =pe_p(A_1,\ldots ,A_p) +\sum_{i=1}^{p-1}(-1)^{i} s_i(A_1,\ldots ,A_p) e_{p-i}(A_1,\ldots ,A_p) \\
 & \equiv p e_p(A_1,\ldots ,A_p)-\sum_{i=1}^{p-1}  ie_{i}(A_1,\ldots ,A_p)e_{p-i}(A_1,\ldots ,A_p)  \bmod p^3 \\
 & =p e_p(A_1,\ldots ,A_p)-\sum_{i=1}^{(p-1)/2}  pe_{i}(A_1,\ldots ,A_p)e_{p-i}(A_1,\ldots ,A_p)  \bmod p^3\\ 
 & \equiv p e_p(A_1,\ldots ,A_p) \bmod p^3,
\end{align*}
as claimed.
\end{proof}

\begin{lemma}\label{lemma2}
Suppose that $n<p$ and that 
$$ \prod_{i=1}^n (1+\alpha_ix)=1+\sum_{i=1}^n e_i(\alpha_1,\ldots ,\alpha_n) x^i \in \mathbb Z [x] $$
with $p\mid e_i(\alpha_1,\ldots ,\alpha_n)$ for all $i=1,\ldots ,n$,  then
$$ \prod_{i=1}^n (1+\alpha_i^p x)=1+\sum_{i=1}^n e_i(\alpha_1^p,\ldots ,\alpha_n^p) x^i \in \mathbb Z [x] $$
has $p^3\mid e_i(\alpha_1^p,\ldots ,\alpha_n^p)$ for all $i=1,\ldots ,n$.
\end{lemma}

\begin{proof} We write $s_j(\alpha_1,\ldots ,\alpha_n)=\sum_{i=1}^n \alpha_i^j$. Observe that 
$$e_k(\alpha_1^p,\ldots ,\alpha_n^p)\equiv e_k(\alpha_1,\ldots ,\alpha_n)^p\equiv 0 \bmod p $$
for $k=1,\ldots ,n$. From the Newton-Girard formulae we
can write
\begin{align} \label{rec} ke_k(\alpha_1^p,\ldots ,\alpha_n^p) & = \sum_{i=1}^{k-1} (-1)^{i-1} e_{k-i}(\alpha_1^p,\ldots ,\alpha_n^p)s_i(\alpha_1^p,\ldots ,\alpha_n^p)+(-1)^{k-1}s_{k}(\alpha_1^p,\ldots ,\alpha_n^p)\\ \nonumber
  & = \sum_{i=1}^{k-1} (-1)^{i-1} e_{k-i}(\alpha_1^p,\ldots ,\alpha_n^p)s_{ip}(\alpha_1,\ldots ,\alpha_n)+(-1)^{k-1}s_{kp}(\alpha_1,\ldots ,\alpha_n),
\end{align}
for $1\leq k\leq n$. For $1\leq k\leq n$ we successively have 
\begin{align*}
s_k(\alpha_1,\ldots ,\alpha_n) &=(-1)^{k-1}ke_k(\alpha_1,\ldots ,\alpha_n) -\sum_{i=1}^{k-1} (-1)^{i} e_i(\alpha_1,\ldots ,\alpha_n)s_{k-i}(\alpha_1,\ldots ,\alpha_n)\\
&\equiv 0 \bmod p,
\end{align*}
while for $k>n$ we have
\be \label{k>n}   s_k(\alpha_1,\ldots ,\alpha_n)= \sum_{i=1}^n (-1)^{i-1} e_{i}(\alpha_1,\ldots ,\alpha_n) s_{k-i}(\alpha_1,\ldots ,\alpha_n)\equiv 0 \bmod p^2,    \ee
and for $k>2n$,
$$  s_k(\alpha_1,\ldots ,\alpha_n)\equiv 0 \bmod p^3. $$
Hence for $k\geq 2$ we have $p^3\mid s_{kp}(\alpha_1,\ldots ,\alpha_n)$, and for $1\leq i<k$ we have $p\mid e_{k-i}(\alpha_1^p,\ldots ,\alpha_n^p)$ and $p^2\mid s_{ip}(\alpha_1,\ldots ,\alpha_n)$  in  \eqref{rec}, so 
$$  ke_k(\alpha_1^p,\ldots ,\alpha_n^p)  \equiv    0 \bmod p^3. $$
For $k=1$ and $p>2n$, we have
$$  e_1(\alpha_1^p,\ldots ,\alpha_n^p)  =s_p(\alpha_1,\ldots ,\alpha_n)\equiv    0 \bmod p^3. $$
Finally, for $k=1$ and $p\leq 2n$, using \eqref{k>n} for any terms $i<p-n$, 
\begin{align*} s_1(\alpha_1^p,\ldots ,\alpha_n^p) & = s_p(\alpha_1,\ldots ,\alpha_n)=\sum_{i=1}^n (-1)^{i-1}e_i(\alpha_1,\ldots ,\alpha_n)s_{p-i}(\alpha_1,\ldots ,\alpha_n) \\
 & \equiv  \sum_{i=p-n}^n (-1)^{i-1}e_i(\alpha_1,\ldots ,\alpha_n)s_{p-i}(\alpha_1,\ldots ,\alpha_n) \bmod p^3\\  
 & \equiv  \sum_{i=p-n}^n (-1)^{i-1}e_i(\alpha_1,\ldots ,\alpha_n)(-1)^{p-i-1}(p-i) e_{p-i}(\alpha_1,\ldots ,\alpha_n) \bmod p^3 \\
 & = -\frac{1}{2}   \sum_{i=p-n}^n (p-i + i) e_i(\alpha_1,\ldots ,\alpha_n)e_{p-i}(\alpha_1,\ldots ,\alpha_n)  \equiv 0 \bmod p^3.  \qedhere
\end{align*}
\end{proof}

\begin{proof}[Proof of Theorem~\ref{main}] For $z =w^j$, $j=1,\ldots ,p-1$, we expand the determinant $D(z)$, 
observing that if we have a term $\ve f_{i_1}(1,z)f_{i_2}(z,z)\cdots f_{i_{p-1}}(z^{p-1},z),$ $\ve=\pm 1,$ with the $i_j$ not all the same, then we also have all the shifts (with the same sign)
 $$\ve f_{i_1}(z^j,z)f_{i_2}(z^{j}z,z)\cdots f_{i_{p-1}}(z^jz^{p-1},z), \;\; j=1,\ldots ,p-1. $$
To see this, observe that the matrices cyclically shifting the rows up one and the columns one to the left have determinant~1.
Hence, gathering those sets of $p$ terms in $B(z)$ and the single  terms with the ${i_j}$ all the same in $A(z)$, we have 
$$ D(z) = A(z) + B(z) $$
with, applying Lemma~\ref{lemma1} at the last step, 
\begin{align} \label{Acong} A(z) & = \sum_{i=0}^{p-1} f_i(1,z)f_i(w,z)f_i(w^2,z)\cdots f_i(w^{p-1},z) =   \sum_{i=0}^{p-1} f_i(1,z)^p + pt(z) \nonumber\\ 
& = \sum_{i=0}^{p-1} f_i(1,1)^p + pt(1) +p(z-1)t_1(z)+(z-1)^pt_2(z) \\
& = \sum_{i=0}^{p-1}  f_i(1,1)f_i(w,1)f_i(w^2,1)\cdots f_i(w^{p-1},1)  +p(z-1)t_1(z)+(z-1)^pt_2(z) \nonumber \\ \nonumber
& \equiv \frac{1}{p} \sum_{y^p=1} \sum_{i=0}^{p-1} f_i(y,1)^p + p(z-1)t_1(z)+(z-1)^pt_2(z) \bmod p^2,
\end{align}
and, with the $\ve (\vec{n})=\pm 1$, 
\begin{align}  \label{Bcong} B(z)  & = \sum_{\vec{n}=(n_1,\ldots ,n_p)\in \mathscr{B}} \sum_{y^p=1}\ve(\vec{n})f_{n_1}(y,z)f_{n_2}(yz,z)\cdots f_{n_p}(yz^{p-1},z)\\ &= pk(z)
  = pk(1)+p(z-1)k_1(z) \nonumber\\
 & =  \sum_{\vec{n} \in \mathscr{B}} \sum_{y^p=1}\ve (\vec{n})f_{n_1}(y,1)f_{n_2}(y,1)\cdots f_{n_p}(y,1)+p(z-1)k_1(z). \nonumber
\end{align}
From the $\mathbb Z_p$  circulant determinants we have
$$ \prod_{x^p=1} F(x,y,1)=  \sum_{i=0}^{p-1} f_i(y,1)^p + p \sum_{\vec{n}\in \mathscr{B}}\ve (\vec{n})f_{n_1}(y,1)\cdots f_{n_p}(y,1), $$
the determinant of the matrix of the same form as $D(z)$, but with the $f_i(z^j,z)$ replaced by $f_i(y,1)$. Hence
$$ D(z)=\frac{1}{p} \sum_{y^p=1}  \prod_{x^p=1} F(x,y,1) + p(z-1)k_2(z)+ (z-1)^pt_2(z) \bmod p^2, $$
giving
$$ M_2= \prod_{j=1}^{p-1} D(w^j) = \left(\frac{1}{p} \sum_{y^p=1}  \prod_{x^p=1} F(x,y,1) \right)^{p-1} \bmod p^2, $$
and
$$ M_2^p \equiv \left(\frac{1}{p} \sum_{y^p=1}  \prod_{x^p=1} F(x,y,1) \right)^{p(p-1)} \bmod p^3. $$
Writing
$$ \prod_{x^p=1} F(x,y,1)= c_0+c_1y+ \cdots + c_{p-1}y^{p-1}   \bmod y^p-1, $$
we have
\begin{equation}\label{eqnM2}
\frac{1}{p} \sum_{y^p=1}  \prod_{x^p=1} F(x,y,1) = c_0,\;\;\; M_2^p\equiv c_0^{p(p-1)} \bmod p^3.
\end{equation}
We note that 
\begin{align*}
 \prod_{x^p=1} F(x,y,1)  & = F(1,y,1)^p +ph(y)= F(1,y^p,1)^p + ph_2(y)\\
 & =F(1,1,1)^p + ph_2(y) \bmod y^p-1,
\end{align*}
so that
\begin{equation}\label{eqnc0F}
c_0\equiv F(1,1,1)^p \bmod p \quad\mathrm{and}\quad p\mid c_1,\ldots ,c_{p-1}.
\end{equation}
If  $p\mid c_0$ then $p \mid F(1,1,1)$  and $p\mid M_2$,  so certainly $p^3\mid M$ and the congruence is trivial (a more precise result for this case is established in Theorem~\ref{pdiv}).
So suppose  that $p\nmid c_0$  and write
$$  c_0+c_1y+\cdots + c_{p-1}y^{p-1}  =c_0 \left(1+d_1y+\cdots + d_{p-1}x^{p-1}\right) \bmod p^3, \;\;\; p\mid d_1,\ldots ,d_{p-1},$$ 
and
$$ 1+d_1x+\cdots + d_{p-1}x^{p-1}=\prod_{i=1}^{r} (1+\alpha_i x),  \;\;\; p\mid d_1,\ldots ,d_{p-1}, $$
so that by Lemma~\ref{lemma2}
\begin{equation}\label{eqnM1}
M_1= \prod_{y^p=1} \prod_{x^p=1} F(x,y,1) = c_0^p\prod_{y^p=1} \prod_{i=1}^r (1+y\alpha_i) =c_0^p\prod_{i=1}^r (1+\alpha_i ^p) \equiv c_0^p \bmod p^3.
\end{equation}
Combining \eqref{eqnM2}, \eqref{eqnc0F} and \eqref{eqnM1} we conclude
\[
M=M_1 M_2^p \equiv c_0^{p^2} \equiv F(1,1,1)^{p^3} \bmod p^3.  \qedhere
\]
\end{proof}

\begin{proof}[Proof of Theorem~\ref{achieve}]
Suppose that $a$ in $\mathbb N$ has  $p\nmid a$.  We define $g(x)$ and  $h(x)$ in $\mathbb Z[x]$ by
\begin{align*}
(1+y+\cdots + y^{a-1})^p & = a + p g(y) \bmod y^p-1, \\
(a+pg(y))^p & = a^p + p^2 h(y)  \bmod y^p-1.
\end{align*}
With $\Phi_p(x)$ as in \eqref{defPhi}, we take
$$ F(x,y,z) =(1+z+\cdots + z^{a-1}) + g(y)\Phi_p(z) + h(x)\Phi_p(y)\Phi_p(z) + m\Phi_p(x)\Phi_p(y)\Phi_p(z). $$
Then
$$ F(x,y,1)=a+pg(y) + ph(x)\Phi_p(y) + mp \Phi_p(x)\Phi_p(y) $$
and for $j=1,\ldots ,p-1$,
$$ F(x,w^j,1)=a+p g(w^j) =\left(1+w^j+\cdots + w^{j(a-1)} \right)^p$$
is a unit, and 
\begin{align*}
F(x,1,1) & = a+pg(1) +p^2h(x)+mp^2\Phi_p(x) \\
 & =a^p+p^2h(x)+mp^2\Phi_p(x) \\
 & = (1+x+\cdots + x^{a-1})^{p^2} +mp^2\Phi_p(x).
\end{align*}
For $x=w^j$, $j=1,\ldots ,p-1$ this equals the unit $(1+w^j+\cdots +w^{j(a-1)})^{p^2}$ and for $x=1$ it equals $a^{p^2}+mp^3$. Hence $M_1=a^{p^2}+mp^3$.
For $\lambda=w^j$, $j=1,\ldots ,p-1$, we have 
$$\phi_{\lambda}(F)= (1+\lambda +\cdots +\lambda^{a-1})I_p, \;\; D(\lambda)= (1+\lambda +\cdots +\lambda^{a-1})^p, $$
a unit, and $M_2(F)=1$ and $M_{H_p}(F)=a^{p^2}+mp^3$. 
\end{proof}

\begin{proof}[Proof of Theorem~\ref{pdiv}]
 With $u=1-w$ we have that the $f_i(w^j,w^k)\equiv f_i(1,1)$ mod $u$ and hence $D(z)$ is congruent 
mod $u$ to a circulant determinant, the $\mathbb Z_p$ measure of $F(x,1,1)$, congruent to $F(1,1,1)^p$ mod $u$,
while $M_1$ is $F(1,1,1)^{p^2}$ mod $p^2$. Hence if $p\mid M_G(F)$ we must have $p\mid F(1,1,1)$. Being a $\mathbb Z_p^2$ measure, this says that $p^{p+3}\mid M_1$ by Theorem~\ref{Zp}. 

From \eqref{Acong} and \eqref{Bcong} we have
$$ D(z)\equiv    \sum_{i=0}^{p-1} f_i(1,1)^p \equiv F(1,1,1)^p   \equiv 0 \mod u^{p-1} $$
in $\mathbb Z[w]$ and $u^{(p-1)^2}\mid M_2$. Since this is an integer, we conclude that $p^{p-1}\mid M_2$ and $p^{p^2+3}\mid M=M_1 M_2^p.$

Suppose that $p\geq 5$ (the case $p=3$ is dealt with in Theorem~\ref{p=3}). Consider 
$$ F(x,y,z)=p+(A-1)^2(1-x)-(1-y)^2 =f_0(y)+x f_1(y), $$
with $f_0(y)=p+(A-1)^2-(1-y)^2$, $f_1(y)= - (A-1)^2 $, 
where $A$ is the smallest integer $2\leq A \leq p-2$ such that $A^p\neq A$ mod $p^2$.
Note that unless $p$ is a Wieferich prime we can take $A=2$ and in the case of a Wieferich prime we have $(p-2)^p\equiv -2^p \equiv -2 \neq p-2$ mod $p^2$, so there will be such an $A$.

As in the proof of Theorem~\ref{Zp} we get that $p^{p+3}\parallel M_1$.  Observe that
\begin{align*}  \prod_{j=0}^{p-1}f_0(w^j) & =\prod_{j=0}^{p-1} (p+(A-1)^2-(1-w^j)^2) \\
  &  \equiv   \prod_{j=0}^{p-1} ((A-1)^2-(1-w^j)^2)
+ p\sum_{i=0}^{p-1}  \prod_{\stackrel{j=0}{j\neq i}}^{p-1} ((A-1)^2-(1-w^j)^2) \bmod p^2\\
 & \equiv  \prod_{j=0}^{p-1} (A-2+w^j)(A-w^j) + p\sum_{i=0}^{p-1} (A-1)^{2(p-1)} \bmod pu^2 \\
 &  \equiv  ((A-2)^p +1)(A^p-1) \bmod pu^2.
\end{align*}
Since this is an integer we actually  have a mod $p^2$ congruence. Since $F$ is linear in $x$ we can use  the formula \eqref{binomial} and from  the minimality of $A$ we find
\begin{align*}  D(z) & \equiv ( (A-2)^p+1)(A^p-1)  -(A-1)^{2p}  \\
 & \equiv (A-1)(A^p-1)-(A-1)^2 = (A-1)(A^p-A)  \bmod p^2\end{align*}
and $p\parallel D(z)$. Hence $p^{p-1}\parallel M_2$ and $p^{(p+3)+p(p-1)} \parallel M_G(F).$
\end{proof}

\section{Some infinite non-abelian groups}\label{secInfinite}

We conclude with some remarks on analogues of the Lind-Mahler measure for three infinite non-abelian groups, and their relationship to a well-known problem of Lehmer.

\subsection{The  discrete Heisenberg group}
It is tempting to  think of the discrete  Heisenberg group $H$ as the limit of $H_p$ as $p\rightarrow \infty$. 
That is, suppose $F(x,y,z)\in \mathbb Z[x,x^{-1},y,y^{-1},z,z^{-1}]$ is a finite sum of the form
$$ F(x,y,z) = \sum_{i,j,k=-\infty}^{\infty} a_{ijk} x^i y^j z^k, $$
thought of as an element in a polynomial ring with non-commuting variables satisfying $xz=zx$, $yz=zy$, but $yx=xyz$.
This suggests defining
$$ m_{H}(F)=\lim_{p\rightarrow \infty} m_{H_p}(F) $$
when this makes sense.
For the special case $F(x,y,z)=f_0(y,z)+x^kf_k(y,z)$ we have from \eqref{binomial}
$$ m_{H_p}(F) = \frac{1}{p} \log \prod_{y^p=1} |f_0(y,1)^p+f_k(y,1)^p|^{1/p^2} + \log \prod_{\stackrel{z^p=1}{z\neq 1}} \left|\prod_{y^p=1} f_0(y,z) + \prod_{y^p=1}f_k(y,z)\right|^{1/p^2}. $$
Assuming non-vanishing, the first term should tend to~0 as $p\rightarrow \infty$.
It seems reasonable then to define the $H$ measure of $F$ by
$$ m_{H}(F) = \int_0^1 \max\left\{ \int_0^1 \log |f_0(e^{2\pi i \theta},e^{2\pi i \xi})|\,d\theta, \int_0^1 \log |f_k(e^{2\pi i \theta},e^{2\pi i \xi})|\,d\theta  \right\} d\xi. $$
This seems consistent with the Heisenberg measure of Lind and Schmidt \cite[Theorem 8.9]{Lind2}, who define the measure as an entropy.

\subsection{The group of rotations and reflections of the plane}

Recall the group  of symmetries of the regular $n$-gon, the dihedral group of order $2n$:
$$ D_{2n} := \langle X,Y\; :\; X^n=1, Y^2=1, XY=YX^{-1}\rangle. $$
In \cite{dihedral} the $D_{2n}$ measure of a polynomial $F(x,y)=f(x)+yg(x)\in\mathbb Z [x,y]$, viewed as  a non-abelian polynomial ring with $x,y$ satisfying the group relations, was defined by
$$ m_{D_{2n}}(F)=\frac{1}{2n} \log |M_{D_{2n}}(F)|,\;\;\; M_{D_{2n}}(F) = \prod_{x^n=1} \left( f(x)f(x^{-1}) - g(x)g(x^{-1}) \right), $$
with $M_{D_{2n}}(F)$ corresponding to the appropriate  group determinant.
Letting $n\to\infty$ (but ignoring cases of vanishing at $x=1$ and other $n$th roots of unity), this suggests a measure  for the group of rotations and reflections of the circle \cite[\S5.5]{Serre},
$$ D_{\infty}=\langle  r_{\alpha}, s \; : \; \alpha\in \mathbb R/\mathbb Z,  \; s^2=1,\;  r_{\alpha}r_{\beta}=r_{\alpha +\beta  \bmod 1}, \;  sr_{\alpha}=r_{-\alpha} s \rangle, $$
where the $r_{\alpha}$ are rotations of angle $2\pi \alpha$ about the origin and $s$ is a reflection.
Namely, we define
\begin{align*}  m_{D_{\infty}}(F) & =\frac{1}{2} \int_{0}^1 \log| f(e^{2\pi i\alpha})f(e^{-2\pi i \alpha}) -  g(e^{2\pi i \alpha})g(e^{-2\pi i\alpha})|\,d\alpha \\
 & = m\left( \sqrt{|f(x)|^2-|g(x)|^2}\right) \end{align*}
for $F(x,y)=f(x)+yg(x) \in \mathbb Z[x,x^{-1},y],$ viewed as an element in the group ring $\mathbb Z[\text{Dih}_{\infty}]$  for the infinite dihedral group
$$ \text{Dih}_{\infty}=\langle x,y\; | \; xy=yx^{-1}, \; y^2=1\rangle. $$
Here, $m(F)$ denotes the traditional  logarithmic Mahler measure for polynomials with complex coefficients,
\[
m(F):=\int_0^1 \log \abs{F(e^{2\pi i\alpha})}\,d \alpha.
\]
Lind \cite{Lind} viewed $m(F)$ as the $\mathbb R/\mathbb Z$ measure for $F$ in $\mathbb Z[x,x^{-1}]$; this could be thought as arising from the $\mathbb Z_n$ measures  as $n\rightarrow \infty$, with $F$ in the group ring for $\mathbb Z=\langle x \rangle$.

Notice that if Lehmer's problem \cite{Lehmer} has a positive answer, that is, there exists a constant $c>0$ such that  for $F$ 
in $\mathbb Z [x]$ either $m(F)=0$ or $m(F)\geq c$, then for $G=D_{\infty}$ we would have for $F\in\mathbb Z[x,y]$
$$ m_G(F)=0 \quad\mathrm{or}\quad m_G(F) \geq c/2. $$
A strong form of Lehmer's problem asks whether the optimal $c$ is
$$ m(x^{10}+x^9-x^7-x^6-x^5-x^4-x^3+x+1)=\log 1.17628\ldots\,. $$ 
Since choosing $f(x)=x^2-1$, $g(x)=x^5+x^4-1$ yields
$$ f(x)f(x^{-1})-g(x)g(x^{-1})=   x^5+x^4-x^2-x-1-x^{-1}-x^{-2}+x^{-4}+x^{-5},  $$
the truth of this strong form would imply that  $\lambda(D_{\infty})=\log \sqrt{1.17628\ldots}$\,.

\subsection{The group $D_{\infty h}$} In \cite{smallgps} (see also \cite{dicyclic}) the measure was defined for the dicyclic group of order $4n$,
$$ Q_{4n} := \langle X,Y\; :\; X^{2n}=1, Y^2=X^n, XY=YX^{-1}\rangle. $$
These are what  Serre \cite[\S5.4]{Serre} called $D_{nh}$, realised as $D_{2n}\times I$ where $I=\{1,\iota\}$, with
$D_{2n}$ viewed as certain reflections and rotations of $3$-space and $\iota$ as the reflection through the origin.
For a polynomial $F(x,y)=f(x)+yg(x),$  $f,g$ in $\mathbb Z[x]/\langle x^{2n}-1 \rangle$, we have
$$ m_{Q_{4n}}(F)=\frac{1}{4n}\log |M_{Q_{4n}}(F)|, $$
where
$$  M_{Q_{4n}}(F)= \prod_{x^{n}=1} \left( f(x)f(x^{-1}) -  g(x)g(x^{-1}) \right) \prod_{x^{n}=-1} \left(  f(x)f(x^{-1}) + g(x)g(x^{-1}) \right). $$
This suggests that for Serre's \cite[\S5.6]{Serre} group $D_{\infty h}$, the group generated by $D_{\infty}$ and the reflection $\iota$ through the origin, we should define,  for any $f(x)$, $g(x)$ in $\mathbb Z[x,x^{-1}]$, 
\begin{align*} m_{D_{\infty h}}(f(x)+yg(x)) & = \frac{1}{4} \int_{0}^1 \log | f(e^{2\pi i\alpha})f(e^{-2\pi i\alpha}) -  g(e^{2\pi i \alpha})g(e^{-2\pi  i\alpha})|\,d\alpha \\
   & +\frac{1}{4}\int_0^1 \log  | f(e^{2\pi i\alpha})f(e^{-2\pi i \alpha}) + g(e^{2\pi i \alpha})g(e^{-2\pi i\alpha})|\,d\alpha \\
 & = m\left( \sqrt[4]{|f(x)|^4- |g(x)|^4 }\right). \end{align*}
Starting with a different standard form for our polynomials in the group ring $\mathbb Z[Q_{4n}]$, we could instead define our $D_{\infty h}$ measure on
$$ F(x,y)=f_0(x)+yf_1(x)+y^2f_2(x)+y^3f_3(x), \;\; f_i(x)\in \mathbb Z[x,x^{-1}], \; y^4=1, \; xy=yx^{-1}, $$
by taking
\begin{align*}
m_{D_{\infty h}}(F)&=\frac{1}{4} m\left(|f_0(x)+f_2(x)|^2-|f_1(x)+f_3(x)|^2\right)\\
&\qquad + \frac{1}{4} m\left(|f_0(x)-f_2(x)|^2+|f_1(x)-f_3(x)|^2\right).
\end{align*}
In either case, a positive answer to Lehmer's problem would  imply that $m_{D_{\infty h}}(F)=0$ or $m_{D_{\infty h}}(F)\geq c/4$.


\begin{thebibliography}{99}
\bibitem{dihedral}
T.~Boerkoel and C.~Pinner, \textit{Minimal  group determinants and the Lind-Lehmer problem for dihedral groups},  Acta Arith. \textbf{186} (2018), no. 4, 377--395.
MR3879399

\bibitem{Conrad}
K.~Conrad, \textit{The origin of representation theory}, Enseign. Math. (2) \textbf{44} (1998), no. 3--4, 361--392.
MR1659232

\bibitem{Lalin}
O.~T. Dasbach and M.~N. Lal\'{i}n,  \textit{Mahler measure  under variations of the base group,} Forum Math. \textbf{21} (2009), no. 4, 621--637.
MR2541476

\bibitem{pgroups}
D.~De Silva, M.~J. Mossinghoff, V.~Pigno  and  C.~Pinner,  \textit{The Lind-Lehmer constant for certain $p$-groups}, 
Math. Comp. \textbf{88} (2019), no. 316, 949--972.
MR3882290

\bibitem{dilum}
D.~De Silva and  C.~Pinner,  \textit{The Lind-Lehmer constant for $\mathbb Z_p^n$}, Proc. Amer. Math. Soc. \textbf{142} (2014), no.~6, 1935--1941.
MR3182012

\bibitem{Formanek}
E.~Formanek and D.~Sibley, \textit{The group determinant determines the group}, Proc. Amer. Math. Soc. \textbf{112} (1991), no. 3, 649--656.
MR1062831

\bibitem{Frobenius}
F.~G. Frobenius, \textit{\"Uber die Primefactoren der Gruppendeterminante}, Gesammelte Ahhandlungen, Band III, Springer, New York, 1968, pp. 38--77.
MR0235974

\bibitem{Norbert}
N.~Kaiblinger,  \textit{On the Lehmer constant of finite cyclic groups}, Acta 
Arith. \textbf{142} (2010), no.~1, 79--84.
MR2601051

\bibitem{Norbert2}
N.~Kaiblinger, \textit{Progress on Olga Taussky-Todd's circulant problem}, 
Ramanujan J. \textbf{28} (2012), no. 1, 45--60.
MR2914452

\bibitem{Laquer}
H.~T. Laquer, \textit{Values of circulants with integer entries},
in A Collection of Manuscripts Related to the Fibonacci Sequence,
Fibonacci Assoc., Santa Clara, 1980, pp. 212--217.
MR0624127

\bibitem{Lehmer}
D.~H. Lehmer, \textit{Factorization of certain cyclotomic functions}, Ann. of Math. (2) \textbf{34} (1933), no.~3, 461--479.
MR1503118

\bibitem{Lind}
D.~Lind, \textit{Lehmer's problem for compact abelian groups}, Proc. Amer. 
Math. Soc. \textbf{133} (2005), no.~5, 1411--1416.
MR2111966

\bibitem{Lind2}
D.~Lind and K.~Schmidt, \textit{A survey of algebraic actions of the discrete Heisenberg group},
Uspekhi Mat. Nauk \textbf{70} (2015), no. 4(424), 77--142;
translation in Russian Math. Surveys \textbf{70} (2015), no. 4, 657--714.
MR3400570

\bibitem{luck}
W.~L\"{u}ck, \textit{Lehmer's problem for arbitrary groups},
J. Topol. Anal., to appear, 32 pp.
\url{doi.org/10.1142/S1793525321500035} 

\bibitem{Newman2}
M.~Newman, \textit{Determinants of circulants of prime power order}, Linear and  
Multilinear Algebra \textbf{9} (1980), no. 3, 187--191.  
MR0601702

\bibitem{Newman1}
M.~Newman, \textit{On a problem suggested by Olga Taussky-Todd}, Illinois J. Math. 
\textbf{24} (1980), no. 1, 156--158.
MR0550657

\bibitem{dicyclic}
B.~Paudel and C.~Pinner, \textit{Minimal group determinants for dicyclic groups}, Mosc. J. Comb. Number Theory,  to appear. 
arXiv:2102.04536 [math.NT].

\bibitem{Pigno1}
V.~Pigno and  C.~Pinner, \textit{The Lind-Lehmer constant for cyclic groups of 
order less than $892,371,480$}, Ramanujan J. \textbf{33} (2014), no.~2, 295--300.
MR3165542

\bibitem{smallgps}
C.~Pinner and C.~Smyth, \textit{Integer group determinants for small groups}, Ramanujan J. \textbf{51} (2020), no. 2, 421--453.
MR4056860

\bibitem{Serre}
J.-P. Serre, Linear Representations of Finite Groups, Grad. Texts in Math., vol. 42, Springer, New York, 1977.
MR0450380

\end{thebibliography}
\end{document}